\documentclass[12pt]{amsart}
\usepackage{amssymb,latexsym}
\usepackage{amsthm}
\usepackage{amsmath}
\usepackage{amsfonts}
\usepackage{url}
\usepackage{hyperref}

\newtheorem{theorem}{Theorem}
\newtheorem{example}{Example}

\newtheorem{remark}{Remark}

\newtheorem{corollary}{Corollary}
\newtheorem{lemma}{Lemma}
\newtheorem{proposition}{Proposition}
\newtheorem{definition}{Definition}

\def\ppp{{\mathbb{P}}}

\def\qqq{\mathbb{Q}}
\def\rrr{\mathbb{R}}
\def\ttt{\mathbb{T}}
\def\ccc{\mathbb{C}}
\def\zzz{\mathbb{Z}}

\def\Div{{\rm Div}}
\def\Pic{{\rm Pic}}
\def\Prin{{\rm Prin}}
\def\Aut{{\rm Aut}}
\def\suchthat{{\ :\ }}

\begin{document}

\author{David Joyner, Amy Ksir, and Caroline Grant Melles}
\address{David Joyner, Mathematics Department, United States Naval Academy, Annapolis, MD 21402}
\email{wdj@usna.edu}

\address{Amy Ksir, Mathematics Department, United States Naval Academy, Annapolis, MD 21402}
\email{ksir@usna.edu}

\address{Caroline Grant Melles, Mathematics Department, United States Naval Academy, Annapolis, MD 21402}
\email{cgg@usna.edu}

\title[Automorphism Groups on Tropical Curves]{Automorphism Groups on Tropical Curves:  \\
Some Cohomology Calculations}

\date{2010-10-15}
\maketitle

\begin{abstract}

Let $X$ be an abstract tropical curve and let $G$ be a finite
subgroup of the automorphism group of $X$.
Let $D$ be a divisor on $X$ whose
equivalence class is $G$-invariant.
We address the following question:
is there always a divisor $D'$ in the equivalence class of $D$ which is
$G$-invariant?
Our  main result is that the answer
is ``yes'' for all abstract tropical curves.
A key step in our proof is a
tropical analogue of Hilbert's Theorem 90.

\end{abstract}

\section {Introduction} We begin by defining an abstract tropical curve $X$
in terms of star-shaped sets, as a generalization of a
metric graph in which all leaves have infinite length.
Our definition is based on papers of Zhang \cite{Z},
Baker and Rumely \cite{BR}, and Haase, Musiker, and Yu \cite{HMY}.  See also
Mikhalkin and Zharkov \cite{MZ},  Baker and Faber \cite{BF}, and
Richter-Gebert, Sturmfels, and Theobald \cite{RST}.
We define rational functions, divisors,
and divisor classes in this setting, following the conventions of
Mikhalkin and Zharkov \cite{MZ}, Gathmann and Kerber \cite{GK}, and Haase, Musiker, and Yu \cite{HMY}.
We note that the automorphism group of an abstract tropical curve $X$
is necessarily finite unless $X$ is homeomorphic to a circle or a closed interval.

In Section 3 we review basic definitions of group cohomology  and set up two
long exact sequences which will be used to prove our main results. These long exact sequences
give relationships among the cohomology groups of $G$ with coefficients in
the real numbers $\rrr$, the group $M(X)$ of rational functions on $X$,
the group $\Prin(X)$ of principal divisors on $X$, the group $\Div(X)$ of
divisors on $X$, and the Picard group
$\Pic(X)$ of classes of linearly equivalent divisors on $X$.

In Section 4 we use methods
similar to those used in the classical case in Goldstein,
Guralnick, and Joyner \cite{GGJ} to show  that if $G$ is a finite subgroup of
the automorphism group of $X$ then

\begin{enumerate}
\item      $H^1(G,\rrr)=0$,
\item     $H^1(G,M(X))=0$ (Tropical Analogue of Hilbert's Theorem 90),
\item   $H^2(G,\rrr)=0$, and
\item    $H^1(G,\Prin(X))=0$ (a direct consequence of the vanishing of
$H^1(G,M(X))$ and $H^2(G,\rrr)$).
\end{enumerate}
The vanishing of $H^1(G,\rrr)$ implies that every $G$-invariant
principal divisor is the image of
a $G$-invariant rational function.
The vanishing of $H^1(G,\Prin(X))$ gives our main result, which is
that every $G$-invariant divisor class contains
a $G$-invariant divisor.

In Section 5 we give two additional results on group cohomology for abstract
tropical curves.
We show that if $G$ is a finite subgroup of
the automorphism group of $X$ then $H^1(G,\Div(X))=0$ and
$H^2(G,M(X)\otimes \qqq)=0$.  It would be
interesting to know whether $H^2(G,M(X))$ vanishes, since this would be a tropical analogue of Tsen's Theorem.

We conclude in Section 6 with some remarks on invariance in degree 0.

\section{Background on Abstract Tropical Curves}

Let
$\Bbb T$ be the tropical semiring
     $$ \Bbb T = \rrr \cup \{ - \infty \}$$
with the tropical operations
     $$x \oplus y = \max \{x,y\}$$
and
     $$x \odot y = x + y$$
(so tropical multiplication is classical addition).
We follow the conventions of Mikhalkin \cite{M1}, using max rather than min
for tropical addition.

Note that there is no inverse for tropical addition, but that $-\infty$
is a neutral element for tropical addition since
     $$-\infty \oplus x = \max \{-\infty,x\} = x$$
for any $x $ in $\ttt$.

Similarly, $0$ is a neutral element for tropical multiplication since
     $$0 \odot x = 0 + x = x$$
for any $x$ in $\ttt$.
Every element $x$ of $\Bbb T$ except $-\infty$ has an inverse $-x$
under tropical
multiplication.

The topology on  $\ttt$ will be taken to be the topology generated by
all open sets of $\rrr$ plus all sets of the form
$[-\infty, b) = \{-\infty\} \cup (-\infty,b)$ for $b \in \rrr$.  In this topology,
the set $[-\infty,b]$ is compact.

For convenience, we sometimes omit the tropical operators.  For example, a
tropical polynomial
     $$\sum_{i=0}^n a_i x^i,$$
with $a_i \in \Bbb T$ for all $i$,
means
     $$\max\{a_i+ix \}.$$
Thus  a tropical polynomial on $\rrr$ is a piecewise linear function with nonnegative integer slopes, except
when it is identically $-\infty$, i.e.,
except when $a_i=- \infty$ for all $i$.

A tropical polynomial in two variables may be used to define a tropical
curve embedded in $\rrr^2$, whose support is the nonlinear locus of the polynomial.
Embedded tropical curves may also be defined in $\rrr^n$ and in tropical projective space
$\ttt \ppp^n$.  See, e.g., Mikhalkin \cite{M2} and \cite{M3}, Richter-Gebert, Sturmfels, and Theobald \cite{RST},
Speyer and Sturmfels \cite{SS}, and Maclagan and Sturmfels \cite{MS}.
In this paper, however, we are concerned with abstract tropical curves, rather than
embedded curves.

There are several ways to define an abstract tropical curve.
We define an abstract tropical curve in terms of star-shaped sets, as a
generalization of a metric (or metrized) graph in which all leaves have infinite length.
Our definition is based on papers of Zhang \cite{Z},
Baker and Rumely \cite{BR}, and Haase, Musiker, and Yu \cite{HMY}.  See also
Mikhalkin and Zharkov \cite{MZ}, Mikhalkin \cite{M1}, and Baker and Faber \cite{BF}.

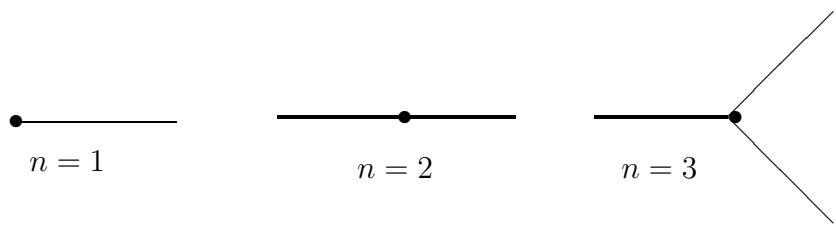
\begin{figure}[h!]
\begin{center}
\begin{tabular}{ccc}
\unitlength=0.620000pt
\begin{picture}(210.00,192.00)(0.00,0.00)
\put(-3.00,46.00){$\bullet$} 

\thinlines
\put(0.00,50.00){\line(1,0){100.00}} 
\put(10.00, 20.00){$n=1$}

\end{picture}

&
\begin{picture}(80,80)(20,20)
\linethickness{0.3mm}
\put(-20,53){\line(1,0){90}}
\put(25.00,50.00){$\bullet$}
\put(10.00, 30.00){$n=2$}

\end{picture}

&
\begin{picture}(80,80)(20,20)
\linethickness{0.3mm}
\put(10,53){\line(1,0){50}}
\put(60.00,50.00){$\bullet$}
\put(60,53){\line(1,1){40}}
\put(60,53){\line(1,-1){40}}
\put(20.00, 30.00){$n=3$}

\end{picture}

\end{tabular}
\caption{Star-shaped set having $n$ arms.}
\end{center}
\label{fig:stars1}
\end{figure}

\begin{definition}
{\rm [Star-shaped set]

A \textbf{star-shaped set} is a set of the form $$S(n,r)=\{ z \in \ccc : z = t e^{\frac{2 \pi i k}{n}} \mbox{ for some } t \in [0, r) \mbox{ and } k \in \zzz \} $$
 where $n$ is a positive integer and $r$ is a positive real number.
 For a fixed $k \in \zzz$ the subset $\{z \in \ccc : z = t e^{\frac{2 \pi i k}{n}} \mbox{ for some } t \in [0, r)\} $ is called an \textbf{arm}; the number of distinct arms is $n$.
 The point at which $z=0$ is called the \textbf{center} of the star-shaped set.  We give each arm of $S(n,r)$ the metric induced from the Euclidean metric on $\ccc$; we give $S(n,r)$ as a whole the path metric and the metric topology.}
 \end{definition}


\begin{definition}
{\rm [Metric topological graph]

Let $X$ be a compact connected topological space such that each point $P \in X$ has a neighborhood homeomorphic to a star-shaped set $S(n_p,r_P)$, where the homeomorphism takes $P$ to the center of the star-shaped set.  The positive integer $n_P$, which is the number of arms of $S(n_P,r_P)$, is called the \textbf{valence} of $P$.
Let $X_0$ be $X \setminus \{P \in X : n_P = 1\}$, i.e., $X$ with its 1-valent points removed.  A \textbf{metric topological graph} is a topological space $X$ as above, with a metric space structure on $X_0$ so that each point $P \in X_0$ has a neighborhood isometric to $S(n_P,r_P)$ for some integer $n_P$ and some positive real number $r_P$.  }
\end{definition}

Note that by compactness, there will be at most finitely many points $P \in X$ with valence $n_P \neq 2$.

\begin{definition}
{\rm [Model of a metric topological graph]

 Suppose that $X$ is a metric topological graph.  Let $V$ be any finite nonempty subset of $X$ such that $V$ contains all of the points with valence $n_P \neq 2$.  Then $X \setminus V$ is homeomorphic to a finite disjoint union of open intervals.  For a given $X$, such a choice of $V$ gives rise to a graph $G(X,V)$ with $V$ as the vertex set and the connected components of $X \setminus V$ as the edge set.  This graph is called a \textbf{model} for $X$.  Unless $X$ is homeomorphic to a circle, we can take $V$ to be $\{ P \in X : n_P \neq 2\}$; we will call the associated graph the \textbf{minimal graph} for $X$.  For any model of $X$, an edge adjacent to a 1-valent vertex is called a \textbf{leaf};  the other edges are called \textbf{inner edges}.  }
 \end{definition}

\begin{definition}
{\rm
[Abstract tropical curve]

Let $X$ be a metric topological graph such that, in every model,
all inner edges have finite length and all leaves have infinite length.
An \textbf{abstract tropical curve} is such a metric topological graph, with a positive integer \textbf{multiplicity} associated to each edge of its minimal graph, or, in the case of a circle, a multiplicity associated to the circle itself.
}
\end{definition}

We will call 1-valent vertices of an abstract tropical curve
\textbf{infinite points}.  All other points are called \textbf{finite points}.
We note that the topology near a 1-valent point is not the metric topology,
because the leaf with its endpoints is compact but has infinite length.  Note also that if $P$ is a 1-valent point,
then there is a homeomorphism
$\iota$ from an interval $[-\infty, b)$ in $\ttt$, where $b \in \rrr$,
to a neighborhood of $P$ in $X$, such that $\iota$ takes $-\infty$ to $P$
and such that the restriction of $\iota$ to $(-\infty, b)$
is an isometry.

\begin{remark}
{\rm
Given a finite graph $G$ with
\begin{enumerate}
\item a finite length associated to each inner edge,
\item infinite length associated to each leaf, and
\item a positive integer multiplicity associated to each edge,
\end{enumerate}
there is a tropical curve (as defined above) with $G$ as a model.}
\end{remark}

\begin{definition}
{\rm
[Automorphisms of abstract tropical curves]

An \textbf{automorphism} $g :  X \rightarrow X$
of an abstract tropical curve $X$ will be defined to be a
map such that
\begin{enumerate}
  \item  $g$ is a homeomorphism on the underlying topological space of $X$,
   \item $g$ is an isometry on $X_0$, and
   \item  $g$ preserves multiplicities.
\end{enumerate}
}
\end{definition}

\begin{remark}
{\rm
If $X$ is not homeomorphic to a circle, then $g$ will be a graph automorphism on the minimal graph for $X$, taking vertices to vertices and edges to edges.}
\end{remark}

The automorphisms
of $X$ form a group, Aut($X$).  In the classical case, Hurwitz's automorphism
theorem gives a bound on the number of automorphisms of a smooth complex projective
algebraic curve of genus $g > 1$.  In the tropical case, we note the following bound.

\begin{theorem}
{\rm
If an abstract tropical curve $X$ has a minimal graph with only one edge, or is homeomorphic to a circle,
then Aut($X$) contains an infinite number of translations.  Otherwise,
the automorphism group Aut($X$) of $X$ is finite, and
moreover if $l$ is the number of leaves of the minimal model for $X$ and $i$ is the
number of inner edges, then
$\Aut(X)$ is contained in the product of symmetric groups $S_l \times S_{2i}$.
}
\end{theorem}

\begin{proof}
In the case where $X$ has a minimal graph with
only one edge, or is homeomorphic to a circle, a translation satisfies all three conditions to be an automorphism.  In any other case, each leaf must have a finite endpoint, and any automorphism of $X$ will map a leaf to another leaf, with the finite endpoint mapping to the finite endpoint and the infinite endpoint mapping to the infinite endpoint.  For each pair of leaves, there is exactly one way to do this preserving the metric on $X_0$.  Similarly,
an automorphism of $X$ must map an inner edge of the minimal graph isometrically to another inner edge of the minimal graph, with the same length and multiplicity.  For each such pair of edges, there are two such isometries.
\end{proof}

\begin{remark}
{\rm
The tropical projective line $\ttt\ppp^1$ is a single edge of infinite length plus its endpoints, and the circle is a genus 1 tropical curve.  See Mikhalkin \cite{M1} for more details.
}
\end{remark}

\begin{example}\label{ex:ex1}
{\rm
Let $n$ be an integer greater than $1$, and let $\Gamma_n$ be the abstract tropical curve consisting of $n$ leaves,
with their endpoints, emanating
from a single $n$-valent point.  Then $\Aut(\Gamma_n) = S_n$.
}
\end{example}

Let $X$ be an abstract tropical curve and let
$f$ be a continuous real-valued function on $X_0$.
Let $P$ be a point in $X_0$ and
let $\iota: S(n_P,r_P) \rightarrow U_P$ be an isometry from a star-shaped set to a neighborhood of $P$,
taking the center of $S(n_P,r_P)$ to $P$.
We will say that $f$ is {\bf piecewise linear at $P$} if
$f\circ \iota$
is piecewise linear on each arm of the star-shaped set.  In other words, for each $k \in \{1, \ldots, n_P \}$, the composition
$[0,r_P) \to \rrr$ given by $t \mapsto f(\iota(t e^{\frac{2 \pi i k}{n_P}}))$ is piecewise linear.
If $f$ is piecewise linear at every point $P \in X_0$,
we will say that it is {\bf piecewise linear on $X$}.
A point of $X_0$ at which $f$ is not linear is called a {\bf singular}
point of $f$.  If $f$ is not locally constant at a point $P$ of valence $n_P > 2$,
then $P$ is a singular point of $f$.
 The slope of $f$,  on any open set on which $f$ is linear,
is well-defined up to sign.  We will say that $f$ is
\textbf{piecewise linear with integer slope}
if $f$ is piecewise linear and has integer slope on any open set on which it
is linear.

Recalling that tropical polynomials on $\rrr$ (if not identically $-\infty$)
are piecewise linear functions
with nonnegative integer slope, and that tropical division
corresponds to classical
subtraction, we define rational functions as follows.

\begin{definition}
{\rm
[Rational functions on an abstract tropical curve]

A \textbf{rational function} on an abstract tropical curve $X$ is
 a continuous real-valued function on $X_0$, the abstract tropical
curve minus its 1-valent points,
which is piecewise linear with integer slope and which has only finitely many
singular points.
Note that a rational
function does not have to be defined at the 1-valent points.

Note also that for the purposes
of this paper, we do not include functions which are identically equal to
$- \infty$ in the set of rational functions.
}
\end{definition}

Let $M(X)$ denote the set of all rational functions on $X$.  Note that
$M(X)$ forms a group with identity element $0$ under tropical multiplication
(classical addition).

Automorphisms of $X$ act on $M(X)$ via their action on $X$.
If $g$
is an automorphism of $X$ and $f$ is a rational function
on $X$, then $gf$ is the rational function given by
     $$gf(P) = f(g^{-1}(P))$$
for every point $P$ in the abstract tropical
curve without infinite points $X_0$, i.e., $gf = f \circ g^{-1} : X_0 \rightarrow \rrr$.

\begin{definition}
{\rm
[Divisors on abstract tropical curves]

 A \textbf{divisor} on an abstract tropical curve $X$
is a finite formal sum of the form
     $$D = \sum_{P \in X} a_P P$$
where, for each $P$, $a_P$ is an integer, and all but finitely many are $0$.
}
\end {definition}

The collection of all divisors on $X$ forms a group $\Div(X)$ under addition,
i.e., the free group over $\zzz$ generated by the points of $X$.

\begin{definition}
{\rm
[Order of a rational function $f$ at a point $P$ of $X$]

Let $f$ be a rational function on an abstract tropical curve $X$.
Essentially, the \textbf{order} of $f$ at a point $P$ of $X$
is the weighted sum of all slopes of $f$ in the direction outward from $P$, for all
edges emanating from $P$, where each edge is weighted according to its multiplicity.
We state this condition more explicitly below.

First, consider the case in which $P$ is not
an infinite point, i.e., which is not
$1$-valent.  Then there is an isometry $\iota$ from a star-shaped set $S(n_P,r_P)$ to a neighborhood of $P$,
taking the center of $S(n_P,r_P)$ to $P$.  Since $f$ is a rational function,
we can restrict the neighborhood and choose a smaller $r_P$, if necessary, so that $f$ is linear on each arm of
$S(n_P, r_P)$.  Thus for each integer $k \in \{1,...,n_P \}$, the composition
$[0,r_P) \to \rrr$ given by $t \mapsto f(\iota(t e^{\frac{2 \pi i k}{n_P}}))$ is linear, with integer slope,  i.e.,
$$ f(\iota(t e^{\frac{2 \pi i k}{n_P}})) = \lambda(k)t + b$$
for some integer $\lambda(k)$ and real number $b$.
We define the order of $f$ at $P$ to be
\[
\text{ord}_P(f) = \sum_{k=1}^{n_P} m(k) \lambda(k),
\]
where $m(k)$ is the multiplicity of the edge which contains the image under $\iota$ of $t e^{\frac{2 \pi i k}{n_P}}$, $0<t<r_P$.

Now suppose that $P$ is a 1-valent point.  Then there is
an isometry $\iota$ from the interval $(-\infty,b)$ in $\rrr$ to a punctured neighborhood of $P$.
Again, since $f$ is a rational function, we can restrict the neighborhood and choose a smaller $b$, if necessary, so that
$f \circ \iota$ is linear with integer slope $\lambda$.  In this case we define
\[
\text{ord}_P(f) = m \lambda,
\]
where again $m$ is the multiplicity of the edge adjacent to $P$.
}
\end{definition}

If a rational function $f$ is linear at a point
$P$, then $\text{ord}_P(f)=0$, so that there are only a finite number of
points $P$ at which $\text{ord}_P(f) \neq 0$ since $f$ has only finitely many singular points.

\begin{definition}
{\rm
[Principal divisors on an abstract tropical curve $X$]

Let $f$ be an element of $M(X)$, i.e., $f$ is a rational
function on the abstract tropical
curve $X$.  We define the \textbf{divisor determined by $f$} to be
     $$\text{div} (f) = (f) = \sum_{P \in X} \text{ord}_P (f)  P.$$
We call such divisors \textbf{principal}.  The set of all principal divisors
forms a subgroup $\Prin (X)$ of $\Div (X)$.
}
\end{definition}

We will say that the \textbf{degree} of a divisor $D=\sum a_P P$ is $\sum a_P$.
Note that the degree of a principal divisor is always 0, since
if $P$ and $Q$ are endpoints of a segment
on which $f$ is linear, the slopes of $f$ emanating from $P$ and $Q$ are
the negative of one another.
(In the special case $P=Q$, i.e., if $f$ is linear on a loop, then $f$ must be constant
on the loop, so the slopes emanating from $P=Q$ on the loop are zero.)
Note also that the degree map is a homomorphism from $\Div(X)$ to $\zzz$.

\begin{definition}
{\rm
[Linear equivalence of divisors]

Divisors $D$ and $D^\prime$  are said to be \textbf{linearly equivalent}
if there is a rational function $f$ such that
\[
D=D^\prime + (f).
\]
}
\end{definition}

\begin{example}\label{ex:nleaf}
{\rm
Let $\Gamma_n$ be the abstract tropical curve
consisting of $n$ leaves, with their endpoints, emanating from a single $n$-valent point $O$.
Let $P$ be any other point on $\Gamma_n$.  Then $P$ and $O$ are
linearly equivalent as divisors, because there is a rational function
with slope 1 on the path from $O$ to $P$ and constant everywhere else.
}
\end{example}

The map div is a group homomorphism
\[
\text{div} : M(X) \rightarrow \Div (X)
\]
from the group of rational functions on $X$ under tropical
multiplication to the group of divisors under addition
since
     $$\text{div} (f_1 \odot f_2) = \text{div} (f_1+f_2) = \text{div} (f_1) + \text{div} (f_2).$$
 The image of the map div is the group $\Prin(X)$ of principal divisors.

The quotient group
\[
\Pic(X)=\Div(X)/\Prin(X)
\]
is called the \textbf{Picard group}.  The elements of the Picard group
are called \textbf{divisor classes}. The divisor class of a divisor $D$
is denoted $[D]$ and consists of all divisors which are linearly equivalent
to $D$.


\begin{example}\label{ex:nleafpic}
{\rm
Let $\Gamma_n$ be as in Example \ref{ex:nleaf}.
Every degree $d$ divisor on $\Gamma_n$ is linearly equivalent to
$dO$, by Example \ref{ex:nleaf}.  Therefore
\[
\Pic(\Gamma_n) \cong \zzz .
\]
}
\end{example}


\section {Background on Group Cohomology}

Let $X$ be an abstract tropical curve
and let $G$ be a finite subgroup of the
automorphism group Aut($X$) of $X$.
Recall that if $X$ has a minimal graph with only one edge, or is homeomorphic to a circle,
then the automorphism group Aut($X$) contains an infinite number of translations.
Otherwise, Aut($X$) is finite, so
every subgroup $G$ of Aut($X$) is necessarily finite.
We review some background material on group cohomology which we will need.
Group cohomology may also be defined in terms of the Ext functor
(see, e.g., Rotman \cite{R} p. 870).
For further information on group cohomology,
we refer to Serre \cite{S}, ch. VII, or the survey
by Joyner \cite{J}.

Let $A$ be a $\zzz[G]$-module.
We can view $\zzz$ as another $\zzz [G]$-module, via the trivial action of $G$ on $\zzz$.
The $0$th cohomology
group of $G$ with coefficients
in $A$ is
     $$H^0(G,A) = \text{Hom}_G(\zzz, A),$$
and is isomorphic to the group $A^G$ of $G$-invariant elements of $A$.
The covariant functor of $G$-invariants,
$A \longmapsto H^0(G,A)\cong A^G$ is left exact.

The \textbf{$1$-cocycles} on $G$ with coefficients in $A$
are defined by
\[
Z^1(G,A) = \{\phi:G\to A\ |\ \forall g_1,g_2\in G,\
\phi(g_1)+g_1 \phi (g_2)= \phi(g_1g_2)\},
\]
the {\textbf{$1$-coboundaries} by
\[
B^1(G,A) = \{ \phi :G\to A\ |\ \exists f\in A \suchthat
\forall g\in G,\ \phi(g)=gf-f\},
\]
and the \textbf{$1$-cohomology} by
\[
H^1(G,A) =Z^1(G,A)/B^1(G,A).
\]
(It is straightforward to check that $B^1(G,A)\subset Z^1(G,A)$.)

The \textbf{$2$-cocycles} on $G$ with coefficients in $A$
are defined by
\[
\begin{array}{l}
Z^2(G,A) =
\{\phi:G\times G\to A\ |\
\forall g_1,g_2,g_3\in G,\\
\qquad \qquad \qquad \qquad  
g_1 \phi(g_2,g_3) - \phi (g_1g_2,g_3)+ \phi(g_1,g_2g_3)- \phi(g_1,g_2)=0\},
\end{array}
\]
the \textbf{$2$-coboundaries}\footnote{It is straightforward to check that
$B^2(G,A)\subset Z^2(G,A)$.} by
\[
\begin{array}{l}
B^2(G,A) = \{\phi:G\times G\to A\ |\ \exists \psi:G\to A \suchthat
\\
\qquad \qquad \qquad \qquad
\forall g_1,g_2\in G,\ \phi(g_1,g_2)=\psi (g_1) + g_1\psi(g_2)- \psi(g_1g_2)\},
\end{array}
\]
and the \textbf{$2$-cohomology} by
\[
H^2(G,A) =Z^2(G,A)/B^2(G,A).
\]

Now we wish to apply this general theory to the case of abstract tropical
curves.  We will describe two short exact sequences.
Lemma \ref{lem:exseq} below is the tropical analogue of the well-known
short exact sequence
\[
1\rightarrow F^\times \rightarrow F(X)^\times
\rightarrow \Prin(X) \rightarrow 0,
\]
for an irreducible non-singular
algebraic curve $X$ over an algebraically closed field $F$,
where $F^\times$ denotes the field minus its zero element and
$F(X)^\times$ denotes the rational functions on $X$ which are
not identically 0.  In the tropical case
we replace $F^\times$ by $\Bbb T^\times = \rrr$ and
$F(X)^\times$ by $M(X)$.

We note that $\rrr$, $M(X)$, and $\Div (X)$ may be
viewed as  $\zzz[G]$-modules.
The action of $G$ on $\rrr$ is the trivial action.
The action of $G$ on $M(X)$ is given by $gf(P)=f(g^{-1}P)$, for $g
\in G$, $f \in M(X)$, and $P \in X$.  The action of $G$
on $\Div(X)$ is the obvious one, i.e.,
if $D=\sum a_P P$ and $g \in G$, then $gD = \sum a_P gP$.
We note that the actions
of $G$ on $M(X)$ and $\Div (X)$ are compatible, since if
$f \in M(X)$ and $g \in G$, then
\begin{align}
\text{div}(gf)&=\sum_{P \in X} \text{ord}_P (gf) P \notag \\
& = \sum_{Q \in X} \text{ord}_{gQ} (f \circ g^{-1})  gQ \notag \\
&=\sum_{Q \in X} \text{ord}_Q(f)  gQ \notag \\
&= g \ \text{div}  (f). \notag
\end{align}
Thus the map $\text{div} : M(X) \rightarrow \Div (X)$ is a
$\zzz[G]$-module homomorphism.

\begin{lemma}
\label{lem:exseq}
{\rm
There is a short exact sequence of $\zzz[G]$-modules,
\[
0\rightarrow \rrr \rightarrow M(X)
\rightarrow \Prin(X) \rightarrow 0.
\]
}
\end{lemma}

\begin{proof}

The order of a rational function $f$ at a point is the sum of the outgoing
slopes. For $f$ to be in the kernel of the map $M(X) \rightarrow
\Prin(X)$, the sum of its outgoing slopes at each point must be equal to $0$.

For $f$ to have order $0$ at every $1$-valent point, $f$ must be constant
on a punctured
open neighborhood of each $1$-valent point (i.e., on a neighborhood of
the vertex minus the vertex itself).  Removing these
open sets gives us a compact set $Y$ on which $f$ is continuous.
Therefore, $f$ must take a  minimum somewhere on $Y$. But at
the point where the minimum is attained, all outgoing slopes are
greater than or equal to $0$. Since the slopes sum to 0, they
must, in fact, all be $0$.
Therefore, $f$ must be constant.
\end{proof}

By the definition of the Picard group, we have a short exact sequence of
$\zzz[G]$-modules.
\[
0\rightarrow \Prin(X) \rightarrow \Div(X)
\rightarrow \Pic(X) \rightarrow 0.
\]

From Lemma \ref{lem:exseq} and the
short exact sequence for $\Pic(X)$ above, we obtain
long exact sequences
\begin{equation}
\label{eqn:longer}
\begin{split}
0\rightarrow H^0(G,\rrr) \rightarrow H^0(G,M(X) )
\rightarrow H^0(G,\Prin(X))\rightarrow \\
H^1(G,\rrr) \rightarrow
H^1(G,M(X))
\rightarrow
H^1(G,\Prin(X)) \rightarrow \\
H^2(G,\rrr)\rightarrow
H^2(G,M(X))\rightarrow H^2(G,\Prin(X)) \rightarrow \ldots
\end{split}
\end{equation}
and
\begin{equation}
\label{eqn:long}
\begin{split}
0\rightarrow H^0(G,\Prin(X)) \rightarrow H^0(G,\Div(X))
\rightarrow H^0(G,\Pic(X)) \rightarrow \\
H^1(G,\Prin(X))
\rightarrow H^1(G,\Div(X))
\rightarrow H^1(G,\Pic(X)) \rightarrow  \\
H^2(G,\Prin(X))
\rightarrow H^2(G,\Div(X))
\rightarrow H^2(G,\Pic(X)) \rightarrow \ldots .
\end{split}
\end{equation}


\section{Proof of Main Result}
Let $X$ be an abstract tropical curve
and let $G$ be a finite subgroup of the
automorphism group  of $X$.
In order to prove our main result, Theorem \ref{theorem:main}, we will
compute various terms of the long exact sequences (\ref{eqn:longer})
and (\ref{eqn:long}).


\begin{lemma}
\label{lem:h1}

$$H^1(G, \rrr) = 0.$$

\end{lemma}

\begin{proof}

Since the action of $G$ on $\rrr$ is trivial,
the condition on 1-cocycles reduces to
     $$Z^1(G, \rrr) = \{ \phi: G \rightarrow \rrr \mid
\forall g_1, g_2 \in G, \phi(g_1) + \phi(g_2) = \phi(g_1g_2) \}.$$
This means that $\phi$ is a homomorphism from the finite group $G$
to $\rrr$, so $\phi$ must be the zero map.
\end{proof}

\begin{corollary}
\label{cor:4}
{\rm
The following is a short exact sequence
\[
0\rightarrow \rrr \to M(X)^G
\rightarrow \Prin(X)^G \rightarrow
0.
\]
In particular, every $G$-invariant principal divisor
is the divisor of a $G$-invariant rational function.
}
\end{corollary}

\begin{proof}

Apply
Lemma \ref{lem:h1} to
the long exact sequence (\ref{eqn:longer}).
\end{proof}

In the case of an algebraic curve, $H^1(G,F(X)^\times )=1$,
by Hilbert's Theorem 90 (see, e.g., Rotman \cite{R} 10.128 and 10.129).
The following theorem is a tropical analogue of Hilbert's Theorem 90.

\begin{theorem}
\label{prop:thrm90}
{\rm Let $X$ be an abstract tropical curve
and let $G$ be a finite subgroup of the
automorphism group of $X$.  Then
$$H^1(G,M(X))=0,$$
where $M(X)$ is the group of rational functions on $X$
under tropical multiplication (classical addition).}
\end{theorem}

\begin{proof}

Pick $\phi \in Z^1(G,M(X))$.
Let $f$ be the tropical sum
     $$f=-{\sum_{g\in G}}^{\text{trop}} \phi(g),$$
i.e., if $P \in X$,
     $$f(P) = -\max_{g \in G} \{\phi(g)(P)\},$$
which is the negative of the tropical average of $\phi$
over $G$.

We compute, for $h \in G$,
\begin{align}
hf(P) &= - \max \{h \phi(g)(P)\} \notag \\
&=- \max \{-\phi(h)(P)+\phi(hg)(P)\} \notag \\
&= \phi(h)(P) + f(P). \notag
\end{align}
Therefore every cocycle is a coboundary.
\end{proof}

\begin{lemma}
\label{lem:h2}

$$H^2(G, \rrr) = 0.$$

\end{lemma}

\begin{proof}

Since the action of $G$ on $\rrr$ is trivial,
\[
\begin{array}{l}
Z^2(G,\rrr) =
\{\phi:G\times G\to \rrr\ |\
\forall g_1,g_2,h\in G,\\
\qquad \qquad \qquad \qquad
\phi(g_2,h) - \phi (g_1g_2,h)+ \phi(g_1,g_2h)- \phi(g_1,g_2)=0\}.
\end{array}
\]
Given $\phi \in Z^2(G,\rrr)$,
define $\psi : G \rightarrow \rrr$ by the classical sum
     $$\psi(g) = \frac 1 {\mid G \mid} \sum_{h \in G}
     \phi(g,h).$$
Then for any $g_1, g_2 \in G$ we have
\begin{align}
\psi(g_1) + g_1 \psi(g_2) - \psi (g_1 g_2) &=
     \psi(g_1) +  \psi(g_2) - \psi (g_1 g_2) \notag \\
&= \frac 1 {\mid G \mid}  \sum_{h \in G}
     \left( \phi(g_1,h) + \phi(g_2,h) - \phi(g_1 g_2,h) \right)\notag \\
&= \frac 1 {\mid G \mid}  \sum_{h \in G}
     \left( \phi(g_1,g_2 h) + \phi(g_2,h) - \phi(g_1 g_2,h) \right)\notag \\
&= \phi(g_1,g_2). \notag
\end{align}
Therefore every 2-cocycle is a 2-coboundary, so $H^2(G, \rrr) = 0$.
\end{proof}

\begin{corollary}\label{cor:h1prin}

$$H^1(G,\Prin(X)) =0.$$

\end{corollary}

\begin{proof}

Apply Proposition \ref{prop:thrm90} and Lemma
\ref{lem:h2} to the long exact sequence (\ref{eqn:longer}).
\end{proof}

The following theorem is our main result and implies
that the answer to the question raised in the
introduction is ``yes'' for all abstract tropical curves.

\begin{theorem}\label{theorem:main}
{\rm
Let $X$ be an abstract tropical curve
and let $G$ be a finite subgroup of the
automorphism group of $X$.
Then the map
\[
\Div (X)^G\to \Pic (X)^G
\]
is surjective, i.e., every $G$-invariant divisor
class contains a $G$-invariant divisor.
}
\end{theorem}

\begin{proof}

Apply Corollary \ref{cor:h1prin} to
 the long exact sequence (\ref{eqn:long}).
\end{proof}

\section{Further Results on Group Cohomology \\
of Abstract Tropical Curves}

Let $X$ be an abstract tropical curve
and let $G$ be a finite subgroup of the
automorphism group of $X$.
Proposition \ref{prop:h1div=0} below is analogous to
a result for algebraic curves which is proven
in Goldstein, Guralnick, and Joyner
\cite{GGJ} using Shapiro's Lemma.  The proof below
is similar but more direct.

\begin{proposition}
\label{prop:h1div=0}

$$H^1(G,\Div(X))=0.$$

\end{proposition}


\begin{proof}  For each $P \in X$, let $G_P$  be the stabilizer
subgroup of $G$ given by $G_P = \{ g \in G | gP=P\}$.
If $h_1$ and $h_2$ are elements of $G$ whose left cosets $\tilde {h_1}$
and $\tilde {h_2}$ in $G/G_P$ are equal, then $h_1P = h_2P$.
Therefore it makes sense to define, for the left coset $\tilde{h}$ of any element $h
\in G$, $\tilde{h}P = hP$.
Let
\[
L_P=\oplus_{\tilde h \in G/G_P} \zzz[\tilde h P].
\]
Let $GX$ be the set of all orbits of points in $X$ and let
$GX/G$ be a complete set of representatives in $X$ of these orbits.
Then $\Div (X)$ is the direct sum of the subgroups $L_P$ for $P$ in
$GX/G$.

Using the characterization of group cohomology as an Ext functor
(see, e.g., Rotman \cite{R} p. 870) and the fact that Ext preserves
direct products in its second argument (see, e.g., Rotman \cite{R} p. 854),
it follows that if $H^1(G,L_P)=0$ for all $P$ in $GX/G$, then
$H^1 (G, \Div (X))=0$.


Next we show that $L_P$ is isomorphic to the co-induced group
     $L^\prime = \text{Coind}_{G_P}^G(\zzz)$
given by
\[
L^\prime=\{f: G \rightarrow \zzz\ |\ f(gh)=f(h)\  \text{for all}
\ g \in G_P \ \text{and} \ h \in G\}.
\]
Each divisor in $L_P$ may be written in the form $\sum_{\tilde h \in G/G_P}
a(\tilde h) \tilde h P$, where $a(\tilde h)$ is an integer for each $\tilde h$.
Given such a divisor,
we define a function $f:G \rightarrow
\zzz$
by $f(h) = a(\tilde {h^{-1}})$.  It is easily checked that $f \in L^\prime$.
If $f \in L^\prime$, and if $\tilde {h_1} =
\tilde {h_2}$, for some $h_1$, and $h_2$ in $G$, then $f(h_1^{-1})=f(h_2^{-1})$,
so we may define $a(\tilde h) = f(h^{-1})$ and the corresponding divisor
$\sum_{\tilde h \in G/G_P} a(\tilde h) \tilde h P$ in $L_P$.

The action of $G$ on $L^\prime$ is given by $gf(h) = f(hg)$
for $g$ and $h$ in $G$.  This action is consistent with the
action of $G$ on $L_P$ and thus $L_P$ and $L^\prime$ are
isomorphic as $\zzz [G]$-modules.

We will show that every 1-cocycle of $G$ in $L^\prime$ is a 1-coboundary.
Suppose that $\phi: G \rightarrow L^\prime$ is in $Z^1 (G, L^\prime)$.
Let $f$ be the map $f: G \rightarrow \zzz$
given by
     $$f(h) = -\phi(h^{-1})(h)$$
for $h \in G$.
First we will show that $f \in L^\prime$ and then that
$\phi(k) = kf-f$ for all $k\in G$, so that $\phi \in B^1 (G,L^\prime)$.

Suppose that $g \in G_P$ and $h \in G$.  We have
\begin{align}
 f(gh) &=
- \phi(h^{-1}g^{-1})(gh) \notag \\
& =    - h^{-1} \phi(g^{-1})(gh) - \phi(h^{-1})(gh)
\qquad \text{since $\phi \in Z^1(G,L^\prime)$} \notag \\
&= - \phi(g^{-1})(g) - \phi(h^{-1})(gh)
\qquad \text{by the action of $G$ on $L^\prime$} \notag \\
&= - \phi(g^{-1})(g) - \phi(h^{-1})(h)
\qquad \text{because
$\phi(h^{-1}) \in L^\prime$ and $g \in G_P$} \notag \\
&=f(g) + f(h). \notag
\end{align}
In particular, the restriction of $f$ to $G_P$ is a homomorphism from $G_P$
to $\zzz$, so $f$ must be $0$ on $G_P$, since $G_P$ is finite.  Therefore
$f(gh) = f(h)$ for all $g \in G_P$ and $h \in G$, so $f$ is in $L^\prime$.

Now we check that $\phi(k) = kf - f$
for all $k \in G$.  For all $h$, $k$, and $l$ in $G$,
\begin{align} \phi(k)(h) &=
     - k\phi(l)(h) + \phi(kl)(h) \qquad \text{since $\phi \in Z^1(G,L^\prime)$} \notag \\
&=- \phi(l)(hk) + \phi(kl)(h) \qquad
\text{by the action of $G$ on $L^\prime$.} \notag
\end{align}
Letting $l = k^{-1}h^{-1}$ gives
\begin{align} \phi(k)(h) &=
- \phi(k^{-1}h^{-1})(hk) + \phi(h^{-1})(h) \notag \\
&=f(hk) - f(h) \notag \\
&= kf(h) - f(h). \notag
\end{align}
Hence $\phi$ is in $B^1(G, L^\prime)$, so $H^1(G,L^\prime) = H^1(G,L_P)=0$.
\end{proof}

In the case of an algebraic curve,
$H^2(G,F^\times(X)))=1$ by Tsen's theorem
(a function field
over an algebraically closed field is a $C^1$ field;
see the Corollaries on pages 96 and
109 of Shatz \cite{Sh}, or \S 4 and \S 7 of chapter X
in Serre \cite{S}). An analogue of Tsen's theorem for tropical curves
would be the computation of $H^2(G,M(X))$.
Such an analogue, if it exists, would be very interesting.  A partial result is as follows.

\begin{lemma}

$$H^2(G,M(X)\otimes \qqq)=0.$$

\end{lemma}

\begin{proof}
We will show that every 2-cocycle of $G$ in $M(X)\otimes \qqq$
is a 2-coboundary.
Suppose that $\phi \in Z^2(G,M(X)\otimes \qqq)$.  Since
(tropical) $\mid G \mid$-th roots exist in $M(X)\otimes \qqq$, we may define a map
$\psi : G \rightarrow M(X)\otimes \qqq$ by the classical sum
     $$\psi(g) = \frac 1 {\mid G \mid} \sum_{h \in G}
\phi(g,h).$$
Then for $g_1, g_2 \in G$ we have
     \begin{align}
\psi (g_1) &+ g_1 \psi (g_2) - \psi(g_1g_2) \notag \\
&=\frac 1 {\mid G \mid} \sum_{h \in G}
\left( \phi(g_1,h) + g_1 \phi(g_2,h) - \phi(g_1 g_2, h) \right)
\notag \\
&= \frac 1 {\mid G \mid} \sum_{h \in G}
\left( \phi(g_1,h) + \phi(g_1g_2,h)-\phi(g_1,g_2 h)+\phi(g_1,g_2)
-\phi(g_1g_2,h)\right) \notag \\
&=\phi(g_1,g_2) + \frac 1 {\mid G \mid} \sum_{h \in G}
\phi (g_1,h) - \frac 1 {\mid G \mid} \sum_{h \in G} \phi (g_1,g_2h) \notag \\
&=\phi(g_1,g_2). \notag
\end{align}
Hence $\phi$ is in $B^2(G,M(X)\otimes \qqq)$, so $H^2(G,M(X)\otimes \qqq)=0$.
\end{proof}


\section{Invariance in Degree 0}

Let $X$ be an abstract tropical curve
and let $G$ be a finite subgroup of the
automorphism group of $X$.
Let $\Pic^0(X)$ be the subgroup
of $\Pic(X)$ consisting of all degree $0$ divisors, i.e., the $\Pic^0(X)$
is the Jacobian variety of $X$.


\begin{remark}
{\rm
Consider the short exact sequence
\[
0 \rightarrow \Prin(X) \rightarrow
\Div^0(X) \rightarrow \Pic^0(X)
    \rightarrow 0.
\]
Note that the map
     $$\Div^0(X)^G \rightarrow \Pic^0(X)^G$$
is surjective, as a trivial consequence of our main result.
Thus every $G$-invariant degree zero divisor class contains a $G$-invariant degree zero divisor.
The classical curve case is more complicated.
}
\end{remark}

\begin{remark}
{\rm
Also, by Corollary \ref{cor:h1prin}, the map
 $$H^1(G,\Div^0(X)) \rightarrow H^1(G,\Pic^0(X))$$
is an injection.
}
\end{remark}

\noindent
{\it Acknowledgement}:
The authors would like to thank the anonymous referee for helpful suggestions.

\end{document}